\documentclass[final,onefignum,onetabnum]{siamart171218}

\usepackage{lipsum}
\usepackage{amsfonts}
\usepackage{graphicx}
\usepackage{epstopdf}
\usepackage{algorithmic}
\ifpdf
  \DeclareGraphicsExtensions{.eps,.pdf,.png,.jpg}
\else
  \DeclareGraphicsExtensions{.eps}
\fi

\newsiamremark{remark}{Remark}
\newsiamremark{hypothesis}{Hypothesis}
\crefname{hypothesis}{Hypothesis}{Hypotheses}
\newsiamthm{claim}{Claim}

\headers{Random variable transformed by Chebyshev Polynomials}{Javier Chico V\'{a}zquez, Andrew J. Horning}

\title{The limiting distribution of a random variable transformed by Chebyshev Polynomials}

\author{Javier Chico V\'{a}zquez\thanks{Imperial College London}\thanks{Mathematical Institute, University of Oxford}
\and Andrew J. Horning \thanks{Massachusetts Institute of Technology}}

\usepackage{amsopn}

\makeatletter
\newcommand*{\addFileDependency}[1]{
  \typeout{(#1)}
  \@addtofilelist{#1}
  \IfFileExists{#1}{}{\typeout{No file #1.}}
}
\makeatother

\newcommand*{\myexternaldocument}[1]{%
    \externaldocument{#1}%
    \addFileDependency{#1.tex}%
    \addFileDependency{#1.aux}%
}

\ifpdf
\hypersetup{
  pdftitle={The limiting distribution of a random variable transformed by chebyshev polynomials},
  pdfauthor={J. Chico V\'{a}zquez and A.J. Horning}
}
\fi

\myexternaldocument{ex_supplement}

\begin{document}

\maketitle

\begin{abstract}
  In this paper we present the result of successively applying a Chebyshev polynomial to a continuous random variable. In particular we show that under mild assumptions the limiting distribution will be the same as the weight with respect to which Chebyshev polynomials are orthogonal, namely $\sim (1-t^2)^{-\frac{1}{2}}$.
\end{abstract}

\begin{keywords}
  Chebyshev polynomials, Convergence in distribution
\end{keywords}


\section{Introduction}
This paper investigates the behaviour of a random variable $X$ when it is transformed by Chebyshev polynomials. Mathematically, we are interested in the distribution of $T_k(X)$, in particular on how this behaves as $k\rightarrow\infty$. In particular, we will show that this converges in distribution to a continuous random variable whose density function is the same as the weight function under which Chebyshev polynomials are orthogonal. Further, we show that the limiting distribution is invariant when transformed by $T_k$.

The structure of this paper is the following: First we review necessary concepts in \cref{sec:prelim}. Next we state the main result in \cref{sec:main}, with an interesting asymptotic analysis of the error and the statement for a lemma which serves as a cornerstone for the proof. This is followed by a presentation of the numerical results. In the next section a rigorous proof for the theorem is presented, and the final remarks include results that seem to hold from the numerics but have not been shown formally.. In the final appendix the proof for the lemma is presented. 

\section{Preliminaries}\label{sec:prelim}
We introduce some important definitions and results before presenting the main results. First, for convenience, we define the $k^{th}$ Chebyshev polynomial of the first kind as \cite{alma991000407698701591, alma9955873875801591,alma991519304401591}
\begin{equation}T_k(x)=\cos(k\cos^{-1}(x))\label{eq:cheb_poly}\end{equation}
where $\cos^{-1}=\arccos{}$. Chebyshev polynomials are orthogonal in the following sense: \cite{alma9955873875801591}
\begin{equation}
    \int_{-1}^1 T_k(x)T_l(x)w(x)\text{d}x = 0 \text{ when } k\neq l \text{ and } w(x)=\frac{1}{\sqrt{1-x^2}}
\end{equation}
The integral of a Chebyshev polynomial over $[-1,1]$ is given by \cite{alma99605834401591}
\begin{equation}
    \int_{-1}^1T_k(x)\text{d}x = \begin{cases} \frac{(-1)^k+1}{1-k^2}& k\neq 1 \\ 0 & k = 1 \end{cases} \label{eq:cheby_integral}
\end{equation}

The following identities for the sums of cosines and sines will be useful later. These are standard results from trigonometry \cite{alma99605834401591}.
\begin{subequations}\begin{align}
\sum_{j=1}^n\cos(j x) &=\frac{\sin(nx/2)\cos((n+1)x/2)}{\sin(x/2)}\\\sum_{j=1}^n\sin(j x)&=\frac{\sin(nx/2)\sin((n+1)x/2)}{\sin(x/2)} \label{eq:sum_sines}\end{align}
\end{subequations}

Once the above has been collected we can proceed by stating the main results and the proof.

\section{Main results}
\label{sec:main}
As mentioned previously, we are interested in the distribution of $T_k(X)$ as $k$ gets large. In particular, under mild assumptions, the following fact holds: $T_k(X)$ will converge in distribution to a random variable whose density function is the same as the weight function for Chebyshev polynomials. 

\begin{theorem}
  Suppose $X$ is a continuous random variable with density function $f_X$ on $[-1,1]$. As it is continuous, it can be expanded as a Chebyshev series that converges absolutely \cite{alma991000407698701591,alma991000347246301591}
\begin{equation}
    f_X(x)=\sum_{l=0}^\infty\mu_l T_l(x) \label{eq:cheb_series}
\end{equation}
then $T_k(X)$ converges in distribution to a random variable $\Gamma$ with probability density function as $k\rightarrow \infty$
$$f_\Gamma(z)=\frac{1}{\pi}\frac{1}{\sqrt{1-z^2}}$$
\end{theorem}

This limiting random variable $\Gamma$ is a scaled and shifted Beta distribution with parameters $\alpha=\beta=\frac{1}{2}$. In particular, if $Y\sim \operatorname{Beta}(\alpha=\frac{1}{2},\beta=\frac{1}{2})$, then $\Gamma=2Y-1$.

For a more general $X$ defined on $[a,b]$ we can re-scale the  Chebyshev polynomials on that interval, with a suitably re-defined weight function as well. The theorem rests on the following result, which provides us with an analytic formula for the probability density function of $T_k(X)$

\begin{lemma}\label{lemma}
Let $f_X$ and $F_X$ be the pdf and cdf of $X$, a continuous random variable define on $[-1,1]$. Then the probability density function of $T_k(X)$ when $k=2m$ is even is
\begin{equation}
\begin{split}
    f_k(z) &=\frac{1}{\sqrt{1-z^2}}\sum_{j=1}^{m=\frac{k}{2}} \left\{f_{\Psi_k}(2\pi j -\cos^{-1}(z))+f_{\Psi_k}(2\pi (j-1) +\cos^{-1}(z))\right\}\\
    &= \frac{1}{\sqrt{1-z^2}}S_k(z)
\end{split}
\end{equation}
where $S_k(z)$ is defined as the sum in the top line. When $k=2m+1$ is odd there is an extra term in the sum
\begin{subequations}
\begin{align}
    f_k(z)=&\frac{1}{\sqrt{1-z^2}}\sum_{j=1}^{m=\frac{k-1}{2}} \{f_{\Psi_k}(2\pi j -\cos^{-1}(z))+f_{\Psi_k}(2\pi (j-1) +\cos^{-1}(z))\} \nonumber\\
    +& \frac{1}{\sqrt{1-z^2}}f_{\Psi_k}(2\pi m +\cos^{-1}(z))\nonumber
    \end{align}
\end{subequations}
where $f_{\Psi_k}$ is the pdf of $\Psi_k = k \cos^{-1}(X)$, so that
\begin{displaymath}
f_{\Psi_k}(\psi)=\frac{1}{k}f_X(\cos(\psi/k))\sin(\psi/k)
\end{displaymath}
\end{lemma}

\subsection{Asymptotic analysis of the convergence}
If we expand the expression for the complete sum in the general case, \cref{eq:full_sum}, around $k\rightarrow\infty$ we obtain the following expansion, 
\begin{equation}
\begin{split}
S_k(z) &\approx \frac{1}{\pi}+\frac{1}{k^2}\left(\frac{\pi}{3}-\frac{(\pi - \cos^{-1}(z))^2}{\pi}\right)\left[\mu_0+\sum_{l=2}^\infty \left\{\mu_l\cos\left(\frac{l\pi}{2}\right)^2\right\}\right]+\mathcal{O}\left(\frac{1}{k^4}\right)\\
&= \frac{1}{\pi}+\frac{1}{k^2}\left(\frac{\pi}{3}-\frac{(\pi - \cos^{-1}(z))^2}{\pi}\right)\sum_{l=0}^\infty\mu_{2l}+\mathcal{O}\left(\frac{1}{k^4}\right)
\end{split}
\end{equation}

In particular, we see that the error is quadratic in $\frac{1}{k}$, thus explaining the rapid convergence found numerically in the upcoming sections. Furthermore, only the even Chebyshev moments contribute to the error. 

\section{Numerical Results}
\subsection{Dance around the origin}
For distributions where the bulk of the probability mass is centered at the origin (i.e. Gaussian, Cauchy, ...), the second Chebyshev polynomial $T_2(t)=2t^2-1$ will map most of that probability mass to $T_2(0)=-1$, so the distribution will become skewed to the left. From there on, as $T_k(1)=1$ and $T_k(-1)=(-1)^{k}$, the probability mass will remain mostly on $t=-1$ for $k=3$, but will migrate to $t=1$ at $k=4$, and then oscillate between $-1$ and $1$ changing place every two iterations, but loosing skewedness as $k$ increases, as we know that the distribution eventually converges to $\frac{1}{\pi\sqrt{1-t^2}}$ which is symmetric. This oscillation is visualized in \cref{fig:dance}, where the initial distribution is a Gaussian. 

\begin{figure}[H]
    \centering
    \includegraphics[scale=0.65]{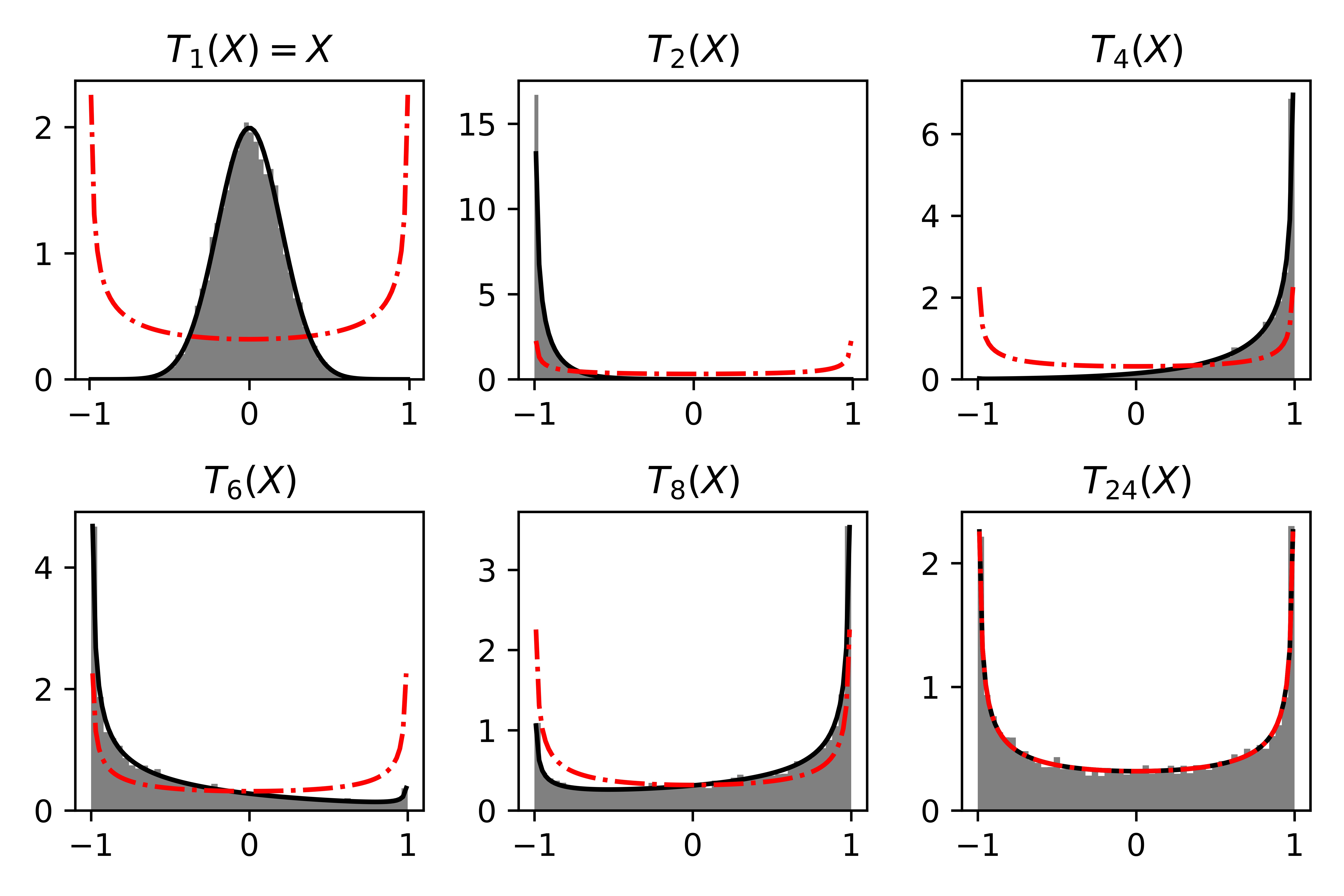}
    \caption{Dance around the origin for an initial Gaussian distribution. The solid black lines represent the analytic solution as given by \cref{lemma} and the dashed red line represents the limiting distribution}
    \label{fig:dance}
\end{figure}

In \cref{fig:dance} we further observe that after only 24 iterations the distribution has essentially converged (very quickly!), indeed, as we will see in the next section this is because the convergence is quadratic in $\frac{1}{k}$, so that the error decreases with $\mathcal{O}(\frac{1}{k^2})$.

\section{Proofs for the main results}

Before proving the general case we will show the proof for the case $X\sim \Gamma$ (in this particular case, instead of convergence we will have that the distribution is invariant, i.e. $\Gamma=T_k(\Gamma)$ in distribution. For this it is necessary to obtain a general formula for the probability density function of $T_k(X)$, denoted as $f_k(z)$. More formally:
\begin{theorem}[The invariant measure] If $X\sim \operatorname{Beta}(1/2,1/2)$ then $T_k(X)\sim \operatorname{Beta}(1/2,1/2)$. This means that the distribution $ \operatorname{Beta}(1/2,1/2)$ is invariant under the transformation $T_k:[-1,1]\longrightarrow[-1,1]$
\end{theorem}
\begin{proof}
  In this case
$$f_X(x)=\frac{1}{\pi\sqrt{1-x^2}}$$
then 
$$f_{\Psi_k}(\psi)=f_\Theta(\psi/k)/k=\frac{1}{k}f_X(\cos(\psi/k))\sin(\psi/k)=\frac{\sin(\psi/k)}{\pi k\sqrt{1-\cos^2(\psi/k)}}=\frac{1}{k\pi}$$
So this is a Uniform distribution!! (which depends on $k$)
Finally, as $f_{\Psi_k}(\psi)$ is constant, the sum is really simple in the case $k$ even.
$$f_k(z)=\frac{1}{\sqrt{1-z^2}}\sum_{j=1}^m f_{\Psi_k}(2\pi j -\cos^{-1}(z))+f_{\Psi_k}(2\pi (j-1) +\cos^{-1}(z))=$$
$$=\frac{1}{\sqrt{1-z^2}}\sum_{j=1}^{k/2}(\frac{1}{k\pi}+\frac{1}{k\pi})=\frac{2}{k\pi\sqrt{1-z^2}}\sum_{j=1}^{k/2}1=\frac{1}{\pi\sqrt{1-z^2}}$$
For $k$ odd we will just have one extra term (but $k$ is one unit greater) so the result is the same and we are done.
\end{proof}

We can now turn our attention to the proof for the general statement. We will do this by first exploring the constraint on the 
Chebyshev moments that arises from normalisation, and we will use that to our advantage to show that the limit in \cref{eq:limit_sum} is indeed $\frac{1}{\pi}$. 

\begin{proof}[General case]
We have that 
\begin{displaymath}
f_X(x)=\sum_{l=0}^\infty\mu_l T_l(x)
\end{displaymath}
But $f_X$ is a pdf so in particular it satisfies
\begin{displaymath}
1 = \int_{-1}^1f_X(x)dx=\sum_{l=0}^\infty\mu_l\int_{-1}^1T_l(x)dx
\end{displaymath}
We can use \cref{eq:cheby_integral} so that the constraint becomes
\begin{equation}
    1 = 2\mu_0+\sum_{l=2}^\infty\mu_l\frac{(-1)^l+1}{1-l^2} \label{eq:constraint_general}
\end{equation}
This will be useful later. We will now proceed as in the simpler cases, by computing the limit in \cref{eq:limit_sum}. 
\begin{displaymath}
f_{\Psi_k}(\psi)=\frac{\sin(\psi/k)}{k}f_X(\cos(\psi/k)=\frac{\sin(\psi/k)}{k}\sum_{l=0}^\infty \mu_lT_l(\cos(\psi/k))
\end{displaymath}
But $T_l(\cos(\psi/k))=\cos\left(\frac{l}{k}\psi\right)$ so that
\begin{displaymath}
f_{\Psi_k}(\psi)=\frac{\sin(\psi/k)}{k}\sum_{l=0}^\infty \mu_l\cos(l\psi/k)
\end{displaymath}

If we focus on $k=2m$ even, then the limit is 
\begin{align}\sum_{j=1}^{k/2} &f_{\Psi_k}(2\pi j -\cos^{-1}(z))+f_{\Psi_k}(2\pi (j-1) +\cos^{-1}(z))=\nonumber\\
= \frac{1}{k}\sum_{j=1}^{k/2} \sum_{l=0}^\infty&\mu_l \sin(x_k j-\beta_k)\cos(l(x_kj -\beta_k)) +\sin((x_k(j-1) +\beta_k)\cos(l(x_k (j-1) +\beta_k)
\end{align}
defining $\beta_k =\cos^{-1}(z)$ as before and $x_k =\frac{2\pi}{k}$. We can exchange the order of summation as the infinite series converges absolutely. 
\begin{displaymath}
 \sum_{l=0}^\infty \mu_l \frac{1}{k}\sum_{j=1}^{k/2}\left\{\sin(x_k j-\beta_k)\cos(l(x_kj -\beta_k)) +\sin(x_k(j-1) +\beta_k)\cos(l(x_k (j-1) +\beta_k))\right\}
\end{displaymath}
thus, we can focus by fixing $l$ and computing the limit of the inner sum as $k$ gets large. We will do this for three separate cases: $l=0,l=1$ and $l\geq2$. For $l=0$ the sum is simplified as the cosines have zero argument so become factors of 1. Thus we have to consider
\begin{displaymath}
 \frac{1}{k}\sum_{j=1}^{k/2}\left\{\sin(x_k j-\beta_k) +\sin(x_k(j-1) +\beta_k)\right\}=\frac{2\cos\left(\frac{\pi-\cos^{-1}(z)}{k}\right)}{k\sin(\pi/k)}
\end{displaymath}
In the limit as $k\rightarrow\infty$ this becomes $\frac{2}{\pi}$. We can see this because $\sin(\pi/k)\sim\frac{\pi}{k}$ and the cosine in the denominator converges to 1 as the argument converges to zero. 

Moving on to the case $l=1$ we can regroup the terms and obtain that the sum (and thus the limit) is identically zero. The summand in this case can be expressed as 
\begin{align*}
\sin(x_k j-\beta_k)\cos(x_kj -\beta_k) +\sin(x_k(j-1) +\beta_k)\cos(x_k (j-1) +\beta_k)=
    \\=\cos\left(\frac{2(\pi-\cos^{-1}(z)}{k}\right)\sin\left(\frac{2\pi(2j-1)}{k}\right)
\end{align*}
So when we sum over $j$ the first term can be moved outside as a constant and the second term, using the formula in \cref{eq:sum_sines} gives us 0. 

Finally, for the case $l\geq 2$, the sum will give us
\begin{align*}
\frac{\cos(l\pi/2)^2}{2k\sin(\pi/k(l+1))\sin(\pi/k(1-l))}[\sin((2\pi+(l-1)\cos^{-1}(z))/k)+\\+\sin((2\pi l +(1-l)\cos^{-1}(z))/k)+\sin((2\pi-(l+1)\cos^{-1}(z))/k)+\\+\sin((-2\pi l+(l+1)\cos^{-1}(z))/k)
]
\end{align*}

As all the sines have an argument that is divided by $k$, in the limit they can be replaced by their argument. This means there is massive cancellation. In particular, in the limit, 
\begin{align*}
    \frac{k^2\cos(l\pi/2)^2}{2k\pi^2(1-l^2)}[4\pi/k]=\frac{1}{\pi}\frac{1+\cos(l\pi)}{1-l^2}=\frac{1}{\pi}\frac{1+(-1)^l}{1-l^2}
\end{align*}
but this is precisely the result in \cref{eq:cheby_integral} multiplied by $\frac{1}{\pi}$. Upon assembling the different case we thus get that in the limit 
\begin{displaymath}
\frac{2}{\pi}\mu_0+0\cdot\mu_1+ \sum_{l=2}^\infty\mu_l \frac{1}{\pi}\frac{1+(-1)^l}{1-l^2}=\frac{1}{\pi}\left[2\mu_0+\sum_{l=2}^\infty\mu_l \frac{1+(-1)^l}{1-l^2}\right]=\frac{1}{\pi}
\end{displaymath}
from the constraint in \cref{eq:constraint_general}, so we are done.

\end{proof}

Before, moving on, we can reassemble the sum before passing onto the limit, for it will be useful later. 
\begin{subequations}
\begin{align}
\label{eq:full_sum}S_k(z)=\frac{2\cos\left(\frac{\pi-\cos^{-1}(z)}{k}\right)}{k\sin(\pi/k)}&+\\
+\frac{\cos(l\pi/2)^2}{2k\sin(\pi/k(l+1))\sin(\pi/k(1-l))}[\sin((2\pi+(l-1)\cos^{-1}(z))/k)+\\+\sin((2\pi l +(1-l)\cos^{-1}(z))/k)+\sin((2\pi-(l+1)\cos^{-1}(z))/k)+\\+\sin((-2\pi l+(l+1)\cos^{-1}(z))/k)
]
\end{align}
\end{subequations}



\section{Final Remarks}
All in all we have shown the main fact for mild assumptions on $X$, in particular that it has a Chebyshev Expansion. In particular, we showed that if the initial distribution is the same as the limiting then it is conserved. Furthermore, we have explored the oscillatory behaviour at the initial stages of the convergence. 

In the future it would be interesting to analytically study the convergence of discontinuous distributions, such as 
$$f_X(x)=\begin{cases}1 & 0<x<1 \\ 0 & \text{else}\end{cases}$$
as our numerical testing strongly suggests the same result applies, as seen in \cref{fig:disc}

\begin{figure}[H]
    \centering
    \includegraphics[scale=0.65]{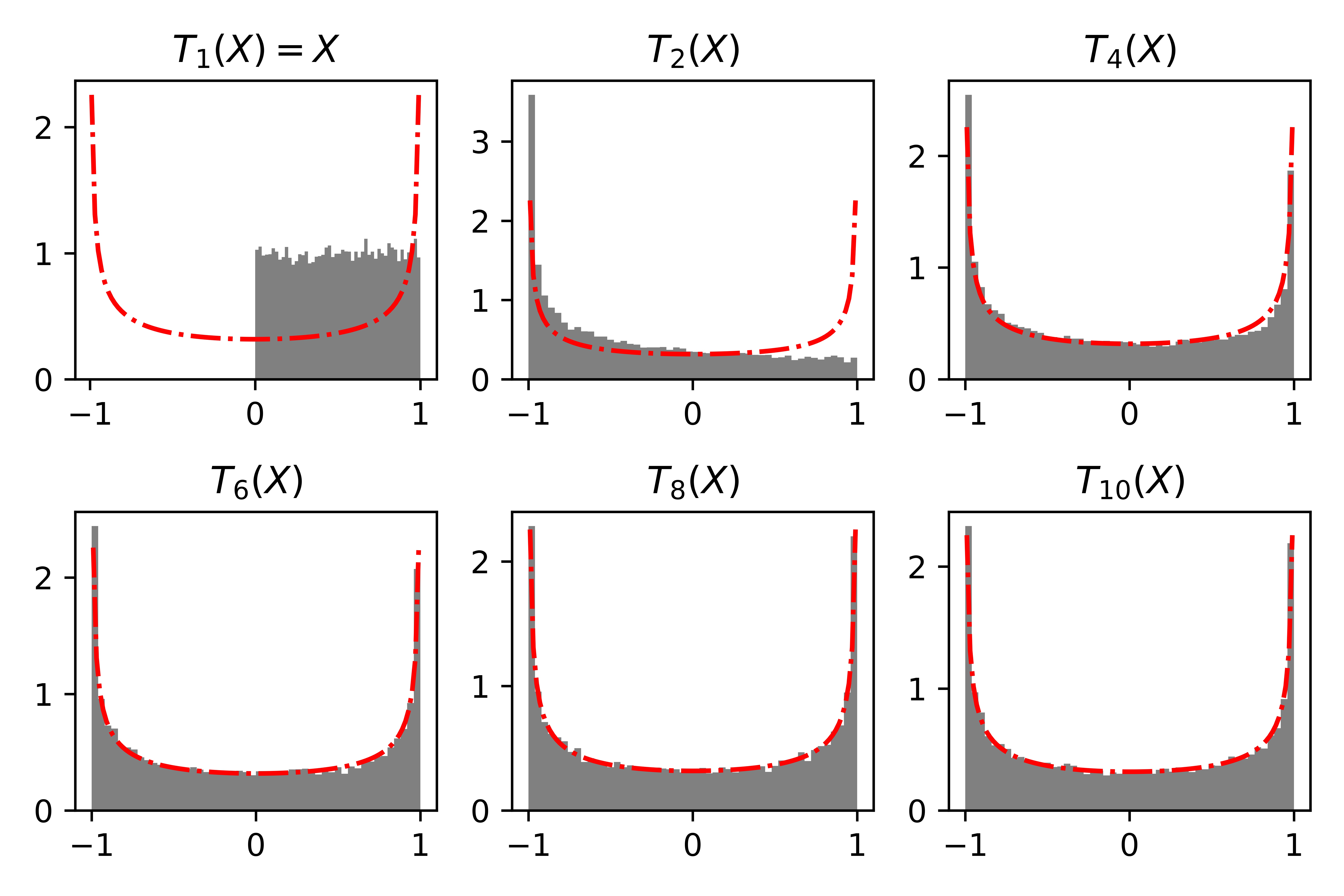}
    \caption{Numerical evidence for the convergence for a discontinuous random variable. }
    \label{fig:disc}
\end{figure}

Furthermore, we would like to explore how this result can be used to accelerate bayesian methods spectral methods like \cite{CadonnaAnnalisa2017Bmmf}

A second possible area for further research is connecting the results presented herein with the theory of Koopman and Perron-Frobenius (Transfer) operators \cite{bruntonModernKoopmanTheory2021}. The main theorem in this paper is the limit for a particular transfer operator applied to the discrete dynamical system with governing equation given by the 3-term recurrence relation for Chebyshev polynomials \cite{alma991000407698701591}:
\begin{equation}
    x_{k+1} = 2 x_0 x_k-x_{k-1}
\end{equation}
where $x_0\in [-1,1]$ is the initial condition.

\section{Appendix}
A constructive proof of \cref{lemma} is first given, as it is essential for the proof of the main theorem. 
\begin{proof}
We want to compute 
$$T_k(X)=\cos(k\cos^{-1}(X))$$
We can motivate computing this quantity in two steps. First, figuring out the distribution of the "inner" part, $k\cos^{-1}(X)$, and then transforming this with the outer cosine. Thus, let $\Theta = \cos^{-1}(X)$. Then 
$$F_\Theta(\theta)= \Pr(\Theta\leq \theta)=\Pr(\cos^{-1}(X)\leq\theta)=\Pr(X\geq \cos(\theta))$$
as the cosine is monotonically decreasing in arccosine [-1,1], so we change the sign of the inequality. Then
$$F_\Theta(\theta)=1-\Pr(X\leq \cos(\theta))=1-F_X(\cos(\theta))$$
and thus, taking a derivative
$$f_\Theta(\theta)= f_X(\cos(\theta))\sin\theta$$
Now define $\Psi_k=k\Theta$. We can easily get the density and cumulative density of this new random variable, as it is just $\Theta$ scaled
$$F_{\Psi_k}(\psi)=F_\Theta(\psi/k)=1-F_X(\cos(\theta/k))$$
and taking a derivative
$$f_{\Psi_k}(\psi)=f_\Theta(\psi/k)/k=\frac{1}{k}f_X(\cos(\psi/k))\sin(\psi/k)$$
so we have completed the first step of the process.
Now we can compute, for $T_k(X)=\cos(\Psi_k)$.  However, we will have to be more careful with the transformation. Let $F_k$ be the cdf of $T_k(X)$, and $f_k$ its pdf as mentioned in the statement. Then
$$F_k(z)=\Pr(\cos(\Psi_k)\leq z)$$
Let's look at that inequality in more detail in \cref{fig:cosine_ineq}. For now, assume $k$ is even:
\begin{figure}[H]
    \centering
    \includegraphics[scale=.4]{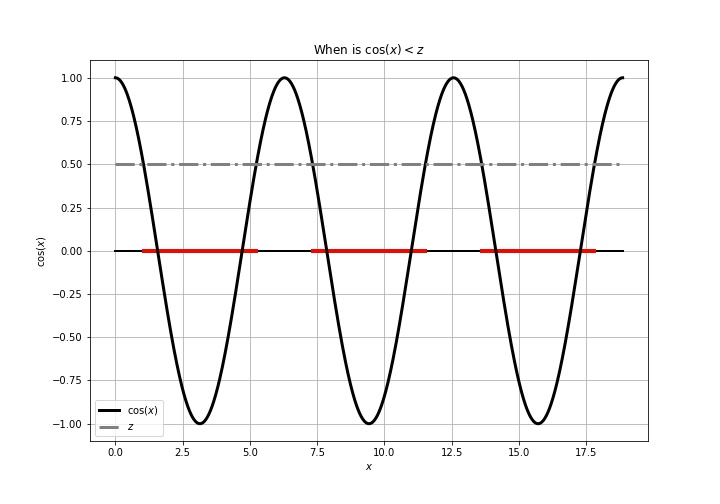}
    \caption{Cosine inequality visualized. It is satisfied on the regions of the $x$-axis highlighted in red.}
    \label{fig:cosine_ineq}
\end{figure}
So we can see it is satisfied in the intervals:

$$\cos^{-1}(z)\leq \Psi_k \leq 2\pi -\cos^{-1}(z)$$
$$2\pi +\cos^{-1}(z)\leq \Psi_k \leq 4\pi -\cos^{-1}(z)$$
$$\vdots$$
$$2\pi (j-1)+\cos^{-1}(z)\leq \Psi_k \leq 2\pi j -\cos^{-1}(z)$$
$$\vdots$$
$$2(k-1)\pi +\cos^{-1}(z)\leq \Psi_k \leq k\pi -\cos^{-1}(z)$$
Let $m=k/2$. Then, as the intervals are disjoint, we have that 
$$F_k(z)=\Pr(\cos(\Psi_k)\leq z)=\sum_{j=1}^m\Pr(2\pi (j-1)+\cos^{-1}(z)\leq \Psi_k \leq 2\pi j -\cos^{-1}(z))$$
We can express the probability for each summand in terms of the cdf of $\Psi_k$:
\begin{equation}
\begin{split}
    &\Pr(2\pi (j-1)+\cos^{-1}(z)\leq \Psi_k \leq 2\pi j -\cos^{-1}(z))\\=& \Pr(\Psi_k \leq 2\pi j -\cos^{-1}(z))-\Pr(\Psi_k \leq 2\pi (j-1) +\cos^{-1}(z))\\
    =& F_{\Psi_k}(2\pi j -\cos^{-1}(z))-F_{\Psi_k}(2\pi (j-1) +\cos^{-1}(z))
\end{split}
\end{equation}
so we get
$$F_k(z)=\sum_{j=1}^m F_{\Psi_k}(2\pi j -\cos^{-1}(z))-F_{\Psi_k}(2\pi (j-1) +\cos^{-1}(z))$$
taking a derivative
\begin{subequations}\begin{align} f_k(z)=\sum_{j=1}^m f_{\Psi_k}(2\pi j -\cos^{-1}(z))\frac{1}{\sqrt{1-z^2}}+f_{\Psi_k}(2\pi (j-1) +\cos^{-1}(z))\frac{1}{\sqrt{1-z^2}}=\\
=\frac{1}{\sqrt{1-z^2}}\sum_{j=1}^m f_{\Psi_k}(2\pi j -\cos^{-1}(z))+f_{\Psi_k}(2\pi (j-1) +\cos^{-1}(z))\end{align} \end{subequations}
Which is what we wanted to show.
\end{proof}
Hence we can already see the main part of the pdf we want to reach in the limit $k\rightarrow \infty$. Therefore, to proof the statement we just need to show that the sum converges (pointwise) to $1/\pi$. Thus we want to show that for a fixed $z$, 
\begin{equation}\lim_{k\rightarrow\infty}\left\{\sum_{j=1}^{k/2} f_{\Psi_k}(2\pi j -\cos^{-1}(z))+f_{\Psi_k}(2\pi (j-1) +\cos^{-1}(z))\right\}= \frac{1}{\pi}\label{eq:limit_sum}\end{equation}

For $k$ odd we just have an extra term in the sum. 
Before proving the general case we will look at two particular examples of great interest: $X$ uniformly distributed and $X\sim \operatorname{Beta}(1/2,1/2)$. For convenience, we define this sum we are taking the limit as $S_k(z)$

\begin{displaymath}
S_k(z)=\sum_{j=1}^{k/2} f_{\Psi_k}(2\pi j -\cos^{-1}(z))+f_{\Psi_k}(2\pi (j-1) +\cos^{-1}(z))
\end{displaymath}

\section*{Acknowledgments}
We would like to acknowledge the MIT Mathematics Department for their support during our UROP.

\bibliographystyle{siamplain}
\bibliography{references}
\end{document}